\documentclass[12pt]{article}

\usepackage{vmargin}
\usepackage{fancyhdr}
\usepackage{ifthen}
\usepackage{graphicx} 
\usepackage{indentfirst}
\usepackage[english]{babel}

\usepackage{fouriernc}
\DeclareMathAlphabet{\mathcal}{OMS}{cmsy}{m}{n} 

\usepackage{amsmath}   
\usepackage{amssymb}   
\usepackage{amsfonts}
\usepackage{amsthm}
\usepackage{mathrsfs}

\usepackage{pst-node} 
\usepackage{pstricks}

\usepackage{tweaklist} 
\renewcommand{\itemhook}{\setlength{\topsep}{1pt}%
\setlength{\itemsep}{-1pt}}

\renewcommand{\enumhook}{\setlength{\topsep}{1pt}%
\setlength{\itemsep}{-1pt}}

\usepackage[colorlinks=true,linkcolor=magenta,citecolor=blue]{hyperref}

\setpapersize{A4}
\setmarginsrb{33mm}{20mm}{33mm}{25mm}{10mm}{6mm}{0mm}{10mm}
\fancypagestyle{plain}{%
\fancyhf{} 
\lfoot[\thepage]{}
\cfoot[]{}
\rfoot[]{\thepage}
}

\newcommand{\sk}{\smallskip}
\newcommand{\mk}{\medskip}
\newcommand{\bk}{\bigskip}
\renewcommand{\emptyset}{\ensuremath{\varnothing}}     

\newcommand{\G}{\mathbf{G}}
\newcommand{\B}{\mathbf{B}}
\newcommand{\U}{\mathbf{U}}
\newcommand{\V}{\mathbf{V}}
\newcommand{\T}{\mathbf{T}}
\renewcommand{\P}{\mathbf{P}}
\renewcommand{\L}{\mathbf{L}}

\newlength{\leftlength}
\newlength{\rightlength}
\newlength{\calculskip}
\newcommand{\calculvskip}[1]{%
  \ifthenelse{#1 = 0}{\setlength{\calculskip}{0pt}}{}%
  \ifthenelse{#1 = 1}{\setlength{\calculskip}{\smallskipamount}}{}%
  \ifthenelse{#1 = 2}{\setlength{\calculskip}{\medskipamount}}{}%
  \ifthenelse{#1 = 3}{\setlength{\calculskip}{\bigskipamount}}{}%
  \ifthenelse{#1 = 4}{\setlength{\calculskip}{1cm}}{}%
  \vskip\calculskip
}

\newcommand{\leftcentersright}[4][2]{%
        \settowidth{\leftlength}{#2}%
        \settowidth{\rightlength}{#4}%
        \calculvskip{#1}
        \noindent#2\hskip-\leftlength%
        \hfill#3\hfill
        \mbox{}\hskip-\rightlength#4%
        \vskip\calculskip%
        }

\newcommand{\centers}[2][2]{\leftcentersright[#1]{}{#2}{}}
\newcommand{\leftcenters}[3][2]{\leftcentersright[#1]{#2}{#3}{}}

\makeatletter
\def\svhline{%
  \noalign{\ifnum0=`}\fi\hrule \@height2\arrayrulewidth \futurelet
   \reserved@a\@xhline}
   
\def\hlinewd#1{%
\noalign{\ifnum0=`}\fi\hrule \@height #1 %
\futurelet\reserved@a\@xhline}
 \makeatother

\numberwithin{equation}{section}

\newtheorem{prop}[equation]{Proposition}  

\newtheorem{thm}[equation]{Theorem}
\newtheorem{lem}[equation]{Lemma}  
\newtheorem{cor}[equation]{Corollary}

\newtheorem*{thm2}{Theorem}
\newtheorem*{conjHLM}{Conjecture HLM}

\theoremstyle{definition}

\newtheorem{rmk}[equation]{Remark}

\newcommand{\X}{\mathrm{X}}
\newcommand{\Y}{\mathrm{Y}}
\newcommand{\Hc}{\mathrm{H}_c}

\newcommand{\Rgc}{\mathrm{R}\Gamma_c}

\newcommand{\ol}{\mathop{\otimes}\limits^\mathrm{L}}
\newcommand\iso{\ \widetilde \to\ }

\usepackage{pstricks-add}
\usepackage{lscape}
\usepackage{multirow}
\usepackage{rotating}

\begin{document}

\title{Coxeter orbits and Brauer trees II}
\author{Olivier Dudas\footnote{Institut de Math\'ematiques de Jussieu, Paris.}
\footnote{At the time this paper was written,  the author was supported by the EPSRC, Project No EP/H026568/1, by Magdalen College, Oxford and partly by the ANR, Project No JC07-192339.}}

\maketitle

\begin{abstract} The purpose of this paper is to discuss the validity of the assumptions $\mathrm{(W)}$ and $\mathrm{(S)}$ stated in \cite{Du3}, about the torsion in the modular $\ell$-adic cohomology of Deligne-Lusztig varieties associated with Coxeter elements. We prove that both $\mathrm{(W)}$ and $\mathrm{(S)}$ hold except for groups of type E$_7$ or E$_8$.
\end{abstract}

\section*{Introduction}

Let $\G$ be a quasi-simple algebraic group defined over an algebraic closure of a finite field of characteristic $p$. Let $F$ be the Frobenius endomorphism of $\G$ associated with a  rational $\mathbb{F}_q$-structure. The finite group $G = \G^F$ of fixed points under $F$ is called a finite reductive group.

\sk

Let $\ell$ be a prime number different from $p$ and $\Lambda$ be a finite extension of $\mathbb{Z}_\ell$. There is strong evidence that the structure of  the principal $\ell$-block of $G$ is encoded in the cohomology over $\Lambda$ of some Deligne-Lusztig variety. Precise conjectures have been stated in \cite{BMa2} and \cite{BMi2}, and much numerical evidence has been collected. The representation theory of $\Lambda G$ is highly dependent on the prime number $\ell$. 
In \cite{Du3}, we have studied a special case referred to as the \emph{Coxeter case}. The corresponding  primes $\ell$ are those which divide the cyclotomic polynomial $\Phi_h(q)$ where  $h$ is the Coxeter number of $W$. In that situation, it is to be expected that the cohomology of the Deligne-Lusztig variety $\Y(\dot c)$ associated with a Coxeter element $c$ describes the principal $\ell$-block $b\Lambda G$. More precisely, 

\begin{itemize}

\item \emph{Hi\ss-L\"ubeck-Malle conjecture}: the Brauer tree of $b\Lambda G$ (which has a cyclic defect group) can be recovered from the action of $G$ and some power $F^\delta$ of $F$ on the cohomology groups $\Hc^i(\Y(\dot c),\overline{\mathbb{Q}}_\ell)$  \cite{HLM};

\item \emph{Geometric version of Brou\'e's conjecture}: the complex $b\Rgc(\Y(\dot c),\Lambda)$ induces a derived equivalence between the principal $\ell$-blocks of $G$ and the normalizer $N_G(T_c)$ of a torus  of type $c$ \cite{BMa2}.

\end{itemize}

\noindent In \cite{Du3} the author has given a general proof of both of these conjectures, but under some assumptions on the torsion in the cohomology of $\Y(\dot c)$. The weaker assumption concerns only some eigenspaces of the Frobenius:

\centers{\begin{tabular}{cp{13cm}} \hskip-1mm $\mathbf{(W)}$ &  For all minimal eigenvalues $\lambda$ of  $F^\delta$,  the generalized $(\lambda)$-eigenspace of $F^\delta$ on $b\Hc^\bullet(\Y(\dot c),\Lambda)$ is torsion-free. \end{tabular}}

\noindent We call here "minimal" the eigenvalues of $F^\delta$ on the cohomology group in middle degree. If this assumption holds, then we proved in \cite{Du3} that the Brauer tree of the principal block has the expected shape. However, a stronger assumption is needed to obtain the planar embedding of the tree and Brou\'e's conjecture:

\centers{\begin{tabular}{cp{13cm}} \hskip-1mm $\mathbf{(S)}$ &  The $\Lambda$-modules $b\Hc^i(\Y(\dot c),\Lambda)$ are torsion-free. \end{tabular}}

The purpose of this paper is to discuss these assumptions for the Deligne-Lusztig variety $\Y(c)$ and some interesting quotients, and to prove that they are valid in the majority of cases. The main result in this direction is the following:

\begin{thm2} Let $b'$ be the idempotent associated with the principal $\ell$-block of $T_c$. If  the type of $G$ is not  E$_7$ or E$_8$,  then the $\Lambda$-modules $b\Hc^i(\Y(\dot c),\Lambda)b'$ are torsion-free.
\end{thm2}

\noindent Furthermore, we can also include the groups of type E$_7$ and E$_8$ if we assume that we already know the shape of the Brauer tree.  These have been now determined by the author and Raphaël Rouquier \cite{DuRou}. This proves a version of \cite[Theorem A]{Du3} where  $b\Hc^i(\Y(\dot c),\Lambda)$ is replaced by $b\Hc^i(\Y(\dot c),\Lambda)b'$. Note that it does not change the results in \cite{Du3} which are implied by this theorem. Note also that the assumption on $p$ in  \cite[Theorem A]{Du3} can be dropped, since we no longer use generalised Gelfand-Graev representations to study the torsion (unlike in \cite{Du2}). In particular, we obtain a significant number of new cases of the geometric version of Brou\'e's conjecture (see Theorem \ref{3thm2}). We also deduce new planar embeddings of Brauer trees for the groups of type ${}^2$G$_2$, F$_4$ and ${}^2 $F$_4$ (see Theorem \ref{3thm3}).

 \sk
  
Our proof relies on Lusztig's work on the geometry of Deligne-Lusztig varieties associated with Coxeter elements \cite{Lu}. Many constructions that are derived from $\X(c)$, such as remarkable quotients and smooth compactifications, can be expressed in terms of varieties associated with smaller Coxeter elements. This provides an inductive method for finding the torsion in the cohomology of $\X(c)$. A further refinement adapted from \cite{BR2} is then used to lift the method up to $\Y(\dot c)$ and to show that the torsion part of $b\Hc^i(\Y(\dot c),\Lambda)b'$ is necessarily a cuspidal module. This reduces the problem of finding the torsion to the problem of finding where cuspidal composition factors can occur in the cohomology. We prove that these cannot occur outside the middle degree.


 \sk

This paper is organized as follows: the first section presents some preliminaries. We have compiled the basic techniques that are used in the modular Deligne-Lusztig theory. In the following section, we use the geometric results of \cite{Lu} to rephrase the assumption $\mathrm{(S)}$ into a more representation-theoretical condition involving cuspidal modules. The last section is devoted to this problem.  We prove that  the cohomology of $\X(c)$ has cuspidal composition factors in the middle degree only whenever the shape of the Brauer tree is known.

\section{Preliminaries\label{1se}}

In this preliminary section we set up the notation and introduce the main techniques (homological and geometric) that we will use throughout this paper. 
\subsection{Homological methods\label{1se1}}

\noindent \thesubsection.1. \textbf{Module categories and usual functors.} If $\mathscr{A}$ is an abelian category, we will denote by $C(\mathscr{A})$ the category of cochain complexes, by $K(\mathscr{A})$ the corresponding homotopy category and by $D(\mathscr{A})$ the derived category. We shall use the superscript notation $-$, $+$ and $b$ to denote the full subcategories of bounded above, bounded below or bounded complexes. We will always consider the case where $\mathscr{A} = A$-$\mathrm{Mod}$ is the module category over some ring $A$ or the full subcategory $A$-$\mathrm{mod}$ of finitely generated modules. This is actually not a strong restriction, since any small abelian category can be embedded into some module category \cite{Mit}. Since the categories $A$-$\mathrm{Mod}$ and $A$-$\mathrm{mod}$ have enough projective objects, one can define the usual derived bifunctors $\mathrm{RHom}_{A}^\bullet$ and ${\ol}_A$. 

\sk

Let $H$ be a finite group and $\ell$ be a prime number.  We fix an $\ell$-modular system $(K, \Lambda, k)$ consisting of a finite extension $K$ of the field of $\ell$-adic numbers $\mathbb{Q}_\ell$, the integral closure $\Lambda$ of the ring of $\ell$-adic integers in $K$ and the residue field $k$ of the local ring $\Lambda$. We assume moreover that the field $K$ is big enough for $H$, so that it contains the $e$-th roots of unity,  where $e$ is the exponent of $H$. In that case, the algebra $KH$ is split semi-simple. 

\sk

From now on, we shall focus on the case where $A = \mathcal{O} H$, with $\mathcal{O}$ being any ring among $(K,\Lambda,k)$. By studying the modular representation theory of $H$ we mean studying the module categories $\mathcal{O} H$-$\mathrm{mod}$ for various $\mathcal{O}$, and also the different connections between them. In this paper, most of the representations will arise in the cohomology of some complexes and we need to know how to pass from one coefficient ring to another. The scalar extension and $\ell$-reduction have a derived counterpart: if $C$ is any bounded complex of $\Lambda H$-modules we can form $KC = C \otimes_\Lambda K$ and $\overline{C} = kC = C {\ol}_\Lambda k$. Since $K$ is a flat $\Lambda$-module, the cohomology of the complex $KC$ is exactly  the scalar extension of the cohomology of $C$. However this does not apply to $\ell$-reduction, but the obstruction can be related to the torsion.

\begin{thm}[Universal coefficient theorem]\label{1thm1}Let $C$ be a bounded complex of $\Lambda H$-modules. Assume that the terms of $C$ are free over $\Lambda$. Then for all $n \geq 1$ and $i \in \mathbb{Z}$, there exists a short exact sequence of $\Lambda H$-modules

\centers{$0 \longrightarrow \mathrm{H}^i(C)\otimes_\Lambda \Lambda / \ell^n \Lambda  \longrightarrow \mathrm{H}^i\big(C \, {\ol}_\Lambda\, \Lambda / \ell^n \Lambda \big) \longrightarrow \mathrm{Tor}_1^\Lambda(\mathrm{H}^{i+1}(C), \Lambda/\ell^n \Lambda) \longrightarrow 0.$}

\end{thm}

In particular, whenever there is no torsion in $\mathrm{H}^\bullet(C)$  (and the terms of $C$ are still assumed to be torsion-free) then the cohomology of $\overline{C}$ is exactly the $\ell$-reduction of the cohomology of $C$.

\bk

\noindent \thesubsection.2. \textbf{Composition factors in the cohomology.}  Let $C$ be a complex of $kH$-modules and $L$ be a simple $kG$-module. We denote by $P_L$ the projective cover of $L$ in $kG$-$\mathrm{mod}$. We can determine the cohomology groups of $C$ in which $L$ occurs as a composition factor using the following standard result.

\begin{lem}\label{1lem1}Given $i \in \mathbb{Z}$, the following assertions are equivalent:

\begin{enumerate}

\item[$\mathrm{(i)}$] The $i$-th cohomology group of the complex $\mathrm{RHom}_{kH}^\bullet (P_L,C)$ is non-zero;

\item[$\mathrm{(ii)}$] $\mathrm{Hom}_{K^b(kH)}(P_L,C[i])$ is non-zero;

\item[$\mathrm{(iii)}$] $L$ is a composition factor of $\mathrm{H}^i(C)$. 

\end{enumerate}

\end{lem}

\begin{proof} See for example \cite[Section 1.1.2]{Du2}. \end{proof}

The formulation $\mathrm{(i)}$ is particularly adapted to our framework and will be extensively used in Section \ref{3se}. 

\bk

\noindent \thesubsection.3. \textbf{Generalized eigenspaces over $\Lambda$.} Let $M$ be a finitely generated $\Lambda$-module and $f \in \mathrm{End}_\Lambda(M)$. Assume that the eigenvalues of $f$ are in the ring $\Lambda$. For $\lambda \in \Lambda$, we have defined in  \cite[Section 1.2.2]{Du3} the generalized $(\lambda)$-eigenspace $M_{(\lambda)}$ of $f$ on $M$. Here are the principal properties that we shall use:

\begin{itemize}

\item $M_{(\lambda)}$ is a direct summand of $M$ and $M$ is the direct sum of the generalized $(\lambda)$-eigenspaces for various $\lambda \in \Lambda$;

\item if $\lambda$ and $\mu$ are congruent modulo $\ell$ then $M_{(\lambda)} = M_{(\mu)}$;

\item $(kM)_{(\lambda)} := M_{(\lambda)} \otimes_\Lambda k$ is the usual generalized $\bar \lambda$-eigenspace of $\bar f$;

\item $(KM)_{(\lambda)} := M_{(\lambda)} \otimes_\Lambda K$ is the sum of the usual generalized $\mu$-eigenspaces of $f$ where $\mu$ runs over the elements of $\Lambda$ that are congruent to $\lambda$ modulo $\ell$.
\end{itemize}

\subsection{Geometric methods\label{1se2}}

To any quasi-projective variety $\X$ defined over $\overline{\mathbb{F}}_p$ and acted on by $H$, one can associate a classical object in the derived category $D^b(\mathcal{O}H$-$\mathrm{Mod})$, namely the cohomology with compact support of $\X$, denoted by $\Rgc(\X,\mathcal{O})$. It is quasi-isomorphic to a bounded complex of finitely generated $\mathcal{O}H$-modules that are free over $\mathcal{O}$. Moreover, the cohomology complex behaves well with respect to scalar extension and $\ell$-reduction. We have indeed in $D^b(\mathcal{O}H$-$\mathrm{Mod})$:

\centers{$  \Rgc(\X,\Lambda) \, {\ol}_{\Lambda } \, \mathcal{O} \, \simeq \, \Rgc(\X, \mathcal{O}).$}

\noindent In particular, the universal coefficient theorem (Theorem \ref{1thm1}) will hold for $\ell$-adic cohomology with compact support.

\sk 

We give here some quasi-isomorphisms we shall use in Sections \ref{2se} and \ref{3se}. The reader will find references or proofs of these properties in  \cite[Section 3]{BR1}.

\begin{prop}\label{1prop1}Let $\X$ and $\Y$ be two quasi-projective varieties acted on by $H$. Then we have the following isomorphisms in the derived category $D^b(\mathcal{O} H$-$\mathrm{Mod})$:

\begin{itemize} 

\item[$\mathrm{(i)}$] The K\"unneth formula: 

\centers[0]{$\Rgc(\X \times \Y, \mathcal{O}) \, \simeq \, \Rgc(\X , \mathcal{O}) \, \ol\, \Rgc(\Y,\mathcal{O}).$}

\item[$\mathrm{(ii)}$]  The quotient variety $H \backslash \X$ exists. Moreover, if the order of the stabilizer of any point of $\X$ is prime to $\ell$, then 

\centers[0]{$\Rgc(H \backslash \X, \mathcal{O}) \, \simeq \, \mathcal{O} \, {\ol}_{\mathcal{O} H} \, \Rgc(\X,\mathcal{O})$.}

\end{itemize}
\end{prop}

If $N$ is a finite group acting on $\X$ on the right and on $\Y$ on the left, we can form the amalgamated product $\X\times_N \Y$, as the quotient of $\X \times \Y$ by the diagonal action of $N$. Assume that the actions of $H$ and $N$ commute and that the order of the stabilizer of any point for the 
diagonal action of $N$ is prime to $\ell$. Then $\X\times_N \Y$ is an $H$-variety and we deduce from the above properties that

\centers{$ \Rgc(\X\times_N \Y, \mathcal{O} ) \, \simeq \, \Rgc(\X , \mathcal{O} ) \, {\ol}_{\mathcal{O} N}\, \Rgc(\Y,\mathcal{O} )$}

\noindent in the derived category $D^b(\mathcal{O} H$-$\mathrm{Mod})$.

\begin{prop}\label{1prop2}Assume that $\Y$ is an open subvariety of $\X$, stable under the action of $H$. Denote by $\mathrm{Z} = \X \smallsetminus \Y$ its complement. Then there exists a distinguished triangle in $D^b(\mathcal{O} H $-$\mathrm{Mod})$:

\centers{$ \Rgc(\Y,\mathcal{O})\ \longrightarrow \ \Rgc(\X,\mathcal{O}) \ \longrightarrow \ \Rgc(\mathrm{Z},\mathcal{O}) \ \rightsquigarrow$}

\noindent Moreover, if $\Y$ is both open and closed, then this triangle splits.
\end{prop}

Finally, for a smooth quasi-projective variety, Poincar\'e-Verdier duality  \cite{SGA45} establishes a remarkable relation between the cohomology complexes $\Rgc(\X,\mathcal{O})$ and $\mathrm{R}\Gamma(\X,\mathcal{O})$. We shall only need the weaker version for cohomology groups:

\begin{thm}[Poincar\'e duality]\label{1thm2}Let $\X$ be a smooth quasi-projective variety of pure dimension $d$. Then if $\mathcal{O}$ is the field $K$ or $k$, there exists a non-canonical isomorphism of $\mathcal{O}H$-modules

\centers{$ \Hc^i(\X,\mathcal{O})^* \, = \, \mathrm{Hom}_{\mathcal{O}}\big(\Hc^i(\X,\mathcal{O}), \mathcal{O}\big) \, \simeq \, \mathrm{H}^{2d-i}(\X,\mathcal{O}).$}

\end{thm}

\section{General results on the torsion\label{2se}}

We present in this section some general results on the torsion in the cohomology of Deligne-Lusztig varieties associated with Coxeter elements. We are motivated by the study of the principal $\ell$-block of $G$ in the \emph{Coxeter case}, involving only a specific class of prime numbers $\ell$. We will only briefly review the results that we will need about the principal $\ell$-block but a fully detailed treatment of the Coxeter case can be found in \cite{Du3}.

\sk

The problem of finding the torsion in the cohomology of a given variety is a difficult problem. We shall use here all the features of the Deligne-Lusztig varieties associated with Coxeter elements: smooth compactifications, filtrations and remarkable quotients by some finite groups. These will be the principal ingredients to prove that the contribution of the principal $\ell$-block to the torsion in the cohomology is necessarily a cuspidal module (see Corollary \ref{2cor2}). This addresses the problem of finding where cuspidal composition factors can occur in the cohomology. We will discuss this general problem in the next section.

\subsection{Finite reductive groups}

We keep the basic assumptions of the introduction, with some slight modification: $\G$ is a quasi-simple algebraic group, together with an isogeny $F$, some power of which is a Frobenius endomorphism. In other words, there exists a positive integer $\delta$ such that $F^\delta$ defines an $\mathbb{F}_{q^\delta}$-structure on $\G$ for a certain power $q^\delta$ of the characteristic $p$ (note that $q$ might not be an integer). Given an  $F$-stable algebraic subgroup $\mathbf{H}$ of $\G$, we will denote by $H$ the finite group of fixed points $\mathbf{H}^F$.  

\sk

We fix a Borel subgroup $\B$ containing a maximal torus $\T$ of $\G$ such that both $\B$ and $\T$ are $F$-stable.  We denote by  $\U$ (resp. $\U^-$)  the unipotent radical of $\B$ (resp. the opposite Borel subgroup $\B^-$). These groups define a root sytem $\Phi$ with basis $\Delta$, and a set of positive (resp. negative) roots $\Phi^+$ (resp. $\Phi^-$). Note that the corresponding Weyl group $W$ is endowed with an action of $F$, compatible with the isomorphism $W \simeq N_\G(\T)/\T$. Therefore, the image under $F$ of a root is a positive multiple of some other root, which will be denoted by $\phi^{-1}(\alpha)$, defining thus a bijection $\phi : \Phi \longrightarrow \Phi$. Since $\B$ is also $F$-stable, this map preserves $\Delta$ and $\Phi^+$. We will also use the notation $[\Delta/\phi]$ for a set of representatives of the orbits of $\phi$ on $\Delta$.  

\subsection{Review of the Coxeter case\label{2se1}}

Recall that the $p'$-part of the order of $G$ is a product of cyclotomic polynomials $\Phi_d(q)$ for various divisors $d$ of the degrees of $W$ (some precautions must be taken for Ree and Suzuki groups \cite{BMa1}). Therefore, if $\ell$ is any prime number different from $p$, it should divide at least one of these polynomials. Moreover, if we assume that $\ell$ is prime to $W^F$, then there is a unique $d$ such that $\ell \, | \, \Phi_d(q)$. 

\sk

In this paper we will be interested in the case where $d$ is maximal. Since $W$ is irreducible, it corresponds to the case where $d=h$ is the Coxeter number of the pair $(W,F)$. Explicit values of $h$ can be found in \cite{Du3}. For a more precise statement $-$ including the Ree and Suzuki groups $-$ recall that a Coxeter element of the pair $(W,F)$ is a product $c= s_{\beta_1} \cdots s_{\beta_r}$ where $\{\beta_1,\ldots,\beta_r\} = [\Delta/\phi]$ is any set of representatives of the orbits of the simple roots under the action of $\phi$. Then the \emph{Coxeter case} corresponds to the situation where $\ell$ is prime to $|W^F|$ and satisfies:

\begin{itemize}

 \item[$\bullet$] "non-twisted" cases: $\ell$ divides $\Phi_h(q)$; 

 \item[$\bullet$] "twisted" cases (Ree and Suzuki groups): $\ell$ divides the order of $T_c$ for some Coxeter element $c$.

\end{itemize}

\noindent Note that these conditions ensure that the class of $q$ in $k^\times$ is a primitive $h$-th root of unity.

\sk

As in Section \ref{1se1}, the modular framework will be given by  an $\ell$-modular system $(K,\Lambda,k)$, which we require to be big enough for $G$. We denote by $b$ an idempotent associated with be the principal block of $\Lambda G$. With the assumptions made on $\ell$, the $\ell$-component of $\T^{cF}$ is a Sylow $\ell$-subgroup of $G$ and as such is a defect group of $b$.  It will be denoted by $T_\ell$.

\sk

The structure of the block is closely related to the cohomology of the Deligne-Lusztig varieties associated with $c$. Fix an $F^\delta$-stable representative $\dot c$ of $c$ in $N_\G(\T)$ and define the varieties $\Y$ and $\X$ by

\sk

\centers{$ \begin{psmatrix}[colsep=2mm,rowsep=10mm] \Y & = \, \big\{ g\U \in \G / \U \ \big| \ g^{-1}F(g) \in \U \dot c \U \big\} \\
					\X & = \, \big\{ g\B \in \G / \B \ \big| \ g^{-1}F(g) \in \B c \B \big\} 
\psset{arrows=->>,nodesep=3pt} 
\everypsbox{\scriptstyle} 
\ncline{1,1}{2,1}<{\pi_c}>{/ \, \T^{cF}}		
\end{psmatrix}$}

\sk

\noindent where $\pi_c$ denotes the restriction to $\Y$ of the canonical projection $\G/\U \longrightarrow \G/\B$. They are both quasi-projective varieties endowed with a left action of $G$ by left multiplication. Furthermore, $\T^{cF}$ acts on the right of $\Y$ and $\pi_c$ is isomorphic to the corresponding quotient map, so that it induces a $G$-equivariant isomorphism of varieties $\Y/ \T^{cF} \iso \X$. Combining the results of \cite{DeLu}, \cite{Lu} and \cite{BMM} we can parametrize the irreducible characters of the principal $\ell$-block:

\begin{itemize}

\item the non-unipotent characters of $bKG$ are exactly the $\theta$-isotypic components $\Hc^r(\Y,K)_\theta$ where $\theta$ runs over the $F^\delta$-orbits in $\mathrm{Irr} \, T_\ell \smallsetminus \{1_{T_\ell}\}$;

\item the unipotent characters of $bKG$ are the eigenspaces of $F^\delta$ on $\Hc^\bullet(\X,K)$. Each eigenvalue is congruent modulo $\ell$ to $q^{j\delta}$ for a unique integer $j \in \{0,\ldots, h/\delta -1\}$. We denote by $\chi_j$ the corresponding irreducible character. Moreover, each Harish-Chandra series that intersects the block corresponds to a root of unity $\zeta \in K$ and the characters in this series are arranged as follows:

\centers{$ \begin{array}{c|c|c|@{\qquad \cdots \qquad}|c} 
\Hc^r(\X ,K) & \Hc^{r+1}(\X ,K)  & \Hc^{r+2}(\X ,K)  & \Hc^{r+M_\zeta-m_\zeta}(\X ,K) 
\\[4pt] \hline \chi_{m_\zeta} \vphantom{\mathop{A}\limits^n} & \chi_{m_\zeta+1} & \chi_{m_\zeta + 2}&  \chi_{M_\zeta} \\[4pt] \end{array} $}

\end{itemize}

\noindent Furthermore, the distinction "non-unipotent/unipotent" corresponds to the distinction  "exceptional/non-exceptional" in the theory of blocks with cyclic defect groups. The connection is actually much deeper: Hi\ss, L\"ubeck and Malle have observed in \cite{HLM} that the cohomology of the Deligne-Lusztig variety $\X$ should not only give the characters of the principal $\ell$-block, but  also its Brauer tree $\Gamma$:

\begin{conjHLM}[Hi\ss-L\"ubeck-Malle]\label{2conj1}Let $\Gamma^\bullet$  denote the graph obtained from the Brauer tree of the principal $\ell$-block by removing the exceptional node and all edges incident to it. Then the following holds:

\begin{enumerate} 

\item[$\mathrm{(i)}$]  The connected components of $\Gamma^\bullet$ are labeled by the Harish-Chandra series, hence by some roots of unity $\zeta \in K$.

\item[$\mathrm{(ii)}$]  The connected component corresponding to a root $\zeta$ is

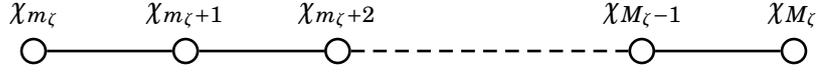
\centers[0]{ \begin{pspicture}(10,1)
  \psset{linewidth=1pt}

  \cnode(0,0.2){5pt}{A}
  \cnode(2,0.2){5pt}{B}
  \cnode(4,0.2){5pt}{C}
  \cnode(8,0.2){5pt}{D}
  \cnode(10,0.2){5pt}{E}
  \ncline{A}{B}  \naput[npos=-0.1]{$\vphantom{\Big(}\chi_{m_\zeta}$} \naput[npos=1.1]{$\vphantom{\Big(}\chi_{m_\zeta+1}$}
  \ncline{B}{C}  \naput[npos=1.1]{$\vphantom{\Big(}\chi_{m_\zeta+2}$}
  \ncline[linestyle=dashed]{C}{D}
  \ncline{D}{E}  \naput[npos=-0.1]{$\vphantom{\Big(}\chi_{M_\zeta-1}$} \naput[npos=1.1]{$\vphantom{\Big(}\chi_{M_\zeta}$}
\end{pspicture}}

\item[$\mathrm{(iii)}$] The vertices labeled by $\chi_{m_{\zeta}}$ are the only nodes connected to the exceptional node.
\end{enumerate}
\end{conjHLM}

The validity of this conjecture has been checked in all cases where the Brauer tree was known, that is for all quasi-simple groups except the groups of type E$_7$ and E$_8$. A  general proof has been exposed  in \cite{Du3}, but under a precise assumption on the torsion in the cohomology of $\Y$: 

\centers{\begin{tabular}{cp{13cm}} \hskip-1mm $\mathbf{(W)}$ &  For all $\zeta \in K$ corresponding to an Harish-Chandra series in the block,  the generalized $(q^{m_\zeta \delta})$-eigenspace of $F^\delta$ on $b\Hc^\bullet(\Y,\Lambda)$ is torsion-free. \end{tabular}}

\noindent This assumption concerns only the "minimal" eigenvalues, that is the eigenvalues of $F^\delta$ on the cohomology of $\X$ in middle degree. Unfortunately, we will not be able to prove it in its full generality for the variety $\Y$, but only for the variety $\X$ (see Proposition \ref{2prop3}). 

\sk

We have investigated in \cite{Du3} the consequences of a stronger assumption, concerning all the eigenvalues of $F^\delta$:

\centers{\begin{tabular}{cp{13cm}} \hskip-1mm $\mathbf{(S)}$ &  The $\Lambda$-modules $b\Hc^i(\Y,\Lambda)$ are torsion-free. \end{tabular}}

\noindent Such an assumption is known to be valid for groups with $\mathbb{F}_q$-rank $1$ (since the corresponding Deligne-Lusztig variety is an irreducible affine curve) and for groups of type A$_n$  \cite{BR2}. The purpose of this paper is to prove that it is actually valid for any quasi-simple group, including E$_7$ and E$_8$ if Conjecture  \hyperref[2conj1]{(HLM)} holds (see Theorem \ref{3thm1}). As a byproduct, we obtain a proof of the geometric version of Brou\'e's conjecture in the Coxeter case. We also deduce the planar embedding of Brauer trees for groups of type ${}^2$G$_2$, F$_4$ and ${}^2 $F$_4$.

\subsection{Torsion and cuspidality\label{2se2}}

Let $I \subset \Delta$ be a $\phi$-stable subset of simple roots. We denote by $\P_I$ the standard parabolic subgroup, by $\U_I$ its unipotent radical and by $\L_I$ its standard Levi complement. The corresponding Weyl group is the subgroup $W_I$ of $W$ generated by the simple reflexions in $I$. All these groups are $F$-stable. One obtains a Coxeter element $c_I$ of $(W_I,F)$ by removing in $c$ the reflexions associated with the simple roots which are not in $I$. Written with the Borel subgroup  $\B_I = \B \cap \L_I$ of $\L_I$, the Deligne-Lusztig variety $\X_I$ associated with $c_I$ is by definition

 \centers{$ \X_I \, = \, \X_{\L_I}(c_I) \, = \, \big\{g\B_I \in \L_I/ \B_I\, \big| \, g^{-1}F(g) \in \B_I c_I \B_I \big\}.$}

Recall  that the Harish-Chandra induction and restriction functors are defined over any coefficient ring $\mathcal{O}$ among $(K,\Lambda,k)$ by

\leftcenters{and}{$ \begin{array}[b]{l} \hphantom{{}^*}\mathrm{R}_{L_I}^G \, : \begin{array}[t]{rcl} \mathcal{O} L_I \text{-}\mathrm{mod} & \longrightarrow & \mathcal{O} G \text{-}\mathrm{mod}  \\ 
N & \longmapsto & \mathcal{O}[G/U_I] \otimes_{\mathcal{O} L_I} N \end{array} \\[20pt]
 {}^*\mathrm{R}_{L_I}^G \, : \begin{array}[t]{rcl} \mathcal{O} G \text{-}\mathrm{mod} & \longrightarrow & \mathcal{O} L_I \text{-}\mathrm{mod}  \\ 
M & \longmapsto & M^{U_I}. \end{array} \end{array}$}

\noindent By the results in \cite{Lu}, the restriction of the cohomology of $\X$ can be expressed in terms of the cohomology of smaller Coxeter varieties, leading to  an inductive approach for studying the torsion. 
 Given the assumptions made on $\ell$ in Section \ref{2se1}, we can prove the following:

\begin{prop}\label{2prop1}If $I$ is a proper $\phi$-stable subset of $\Delta$, then the groups $\Hc^i(\X_I,\Lambda)$ are torsion-free.
\end{prop}

\begin{proof} By induction on the cardinality of $I$, the case $I = \emptyset$ being trivial. Assume that $I$ is non-empty, and let $J$ be any maximal proper $\phi$-stable subset of $I$. By \cite[Corollary 2.10]{Lu} there exists an isomorphism of varieties $(U_J \cap \L_I) \backslash \X_I \simeq \X_J \times \G_m$ which yields the following  isomorphism of $\Lambda$-modules 

\centers{$ {}^* \mathrm{R}_{L_J}^{L_I} \big(\Hc^i(\X_I,\Lambda)\big) \, \simeq \, \Hc^{i-1}(\X_J,\Lambda) \oplus  \Hc^{i-2}(\X_J,\Lambda). $}

\noindent Consequently, the Harish-Chandra restriction to $L_J$ of $\Hc^\bullet(\X_I,\Lambda)_{\mathrm{tor}}$ is a torsion submodule of $\Hc^\bullet(\X_J,\Lambda)$ and hence it is zero by assumption. In other words, any torsion submodule of $\Hc^\bullet(\X_I,\Lambda)_{\mathrm{tor}}$ is a cuspidal $\Lambda L_I$-module. 

\sk

Let $r_I= | I/\phi|$ be the number of $\phi$-orbits in $I$. Following \cite{Lu}, we can define a smooth compactification $\overline\X_I$ of $\X_I$. It has a filtration by closed $L_I$-subvarieties 
 $\overline{\X}_I = D_{r_I}(I) \supset D_{r_I-1}(I) \supset \cdots \supset D_{0}(I) = L_I/B_I$  such that each pair $(D_a(I),D_{a-1}(I))$ leads to the following long exact sequences of  $kL_I$-modules:

\centers{$\cdots \longrightarrow \displaystyle \hskip -5mm \bigoplus_{\begin{subarray}{c}  J \subset I \ \phi\text{-stable}\\[2pt] |J/\phi| = a \end{subarray}} \hskip -2mm\mathrm{R}_{L_J}^{L_I} \big( \Hc^i(\X_J,k) \big) \ \longrightarrow \ \mathrm{H}^i(D_a(I),k) \, \longrightarrow \ \mathrm{H}^i(D_{a-1}(I), k) \ \longrightarrow \cdots$}

\noindent Since $kL_I$ is a semi-simple algebra, we are in the following situation:

\begin{itemize}

\item any cuspidal composition factor of a $kL_I$-module $M$ is actually a direct summand of $M$. Therefore we can naturally define the cuspidal part $M_\mathrm{cusp}$ of $M$ as the sum of all cuspidal submodules;

\item if $J \neq I$ then an induced module $\mathrm{R}_{L_J}^{L_I}(M)$ has a zero cuspidal part;

\item by Poincar\'e duality (see Theorem \ref{1thm2}) $\mathrm{H}^i(\overline\X_I,k)$ is isomorphic to the $k$-dual of the $kL_I$-module $\mathrm{H}^{2r_I-i}(\overline\X_I,k)$.

\end{itemize}

\noindent Consequently, we can argue as in \cite{Lu} and use the following isomorphisms :

\centers{$ \Hc^i(\X_I,k)_\mathrm{cusp} \, \simeq \, \mathrm{H}^i(\overline{\X}_I,k)_\mathrm{cusp} \, \simeq \, \mathrm{H}^{2r_I-i}(\overline{\X}_I,k)_\mathrm{cusp}^*  \, \simeq \, \Hc^{2r_I - i}(\X_I,k)_\mathrm{cusp}^*$}

\noindent to deduce that $\Hc^i(\X_I,k)$ has a zero cuspidal part whenever $i > r_I$. 

\sk

Finally, $\Hc^i(\X_I,\Lambda)_{\mathrm{tor}} \otimes_\Lambda k$ is a cuspidal submodule of $\Hc^i(\X_I,\Lambda)\otimes_\Lambda k$ which is also a submodule of $\Hc^i(\X_I,k)$ by the universal coefficient theorem (see Theorem \ref{1thm1}). Therefore it must be zero if $i \neq r_I$ and $\Hc^i(\X_I,\Lambda)$ is torsion-free. Note that the cohomology group in middle degree is also torsion-free since $\X_I$ is an irreducible affine variety: we have indeed $\Hc^{r_I-1}(\X_I,k) = 0$ so that  $\Hc^{r_I}(\X_I,\Lambda)_{\mathrm{tor}} = 0$ by the universal coefficient theorem.
\end{proof}

\begin{rmk}\label{2rmk1}We have actually shown that the cohomology of a Deligne-Lusztig variety associated with a Coxeter element is torsion-free whenever $\ell$ does not divide the order of the corresponding finite reductive group. One can conjecture that this should hold for any Deligne-Lusztig variety. Note that the assumption on $\ell$ is crucial, otherwise an induced module can have cuspidal composition factors. For example if $\ell$ divides $\Phi_h(q)$, then $\mathrm{R}_T^G(k) = k[G/B]$ has at least one cuspidal composition factor $-$ denoted by $S_0$ in \cite{Du3}. \end{rmk}

As an immediate consequence of Proposition \ref{2prop1} and the isomorphism in \cite[Corollary 2.10]{Lu}, we obtain: 

\begin{cor}\label{2cor1}The torsion part of any group $\Hc^i(\X,\Lambda)$ is a cuspidal $\Lambda G$-module.
\end{cor}

We can therefore reduce  the problem of finding the torsion in the cohomology of $\X$ to the problem of finding where cuspidal composition factors can occur in the $kG$-modules $\Hc^i(\X,k)$. This is the general approach that we shall use throughout this paper. It is justified by the fact that if we assume that the modules $\Hc^i(\X,\Lambda)$ are torsion-free and that Conjecture  \hyperref[2conj1]{(HLM)} holds, then by \cite{Du3} cuspidal composition factors can occur only in $\Hc^r(\X,k)$. 

\subsection{Reduction: from \texorpdfstring{$\Y$}{Y} to \texorpdfstring{$\X$}{X}\label{2se3}}

In this section we shall give conditions on $\X$ ensuring that the principal part of the cohomology of $\Y$ is torsion-free.  Unlike $\Y$, the variety $\X$ has a smooth compactification (the compactification of $\Y$ constructed in \cite{BR4} is only rationally smooth in general). Therefore we can use duality theorems to study precise concentration properties of cuspidal modules in the cohomology of $\X$ (see Sections \ref{2se4} and \ref{3se1}).  
We will not only restrict the cohomology of $\Y$ to the principal block $b$ of $\Lambda G$ but also to the principal block $b'$ of $\Lambda \T^{cF}$. This is not a strong restriction since the complexes $b\Rgc(\Y,\Lambda)$ and $b\Rgc(\Y,\Lambda)b'$ are conjecturally isomorphic in $K^b(\Lambda G$-$\mathrm{Mod})$. Note that this is already the case when the coefficient ring is $K$ (see \cite[Theorem 5.24]{BMM}). By Proposition \ref{1prop1} we have

\centers{$ \Rgc(\Y,\Lambda) b' \, \simeq \,  \Rgc(\Y,\Lambda)\otimes_{\Lambda T_{\ell'}} \Lambda \, \simeq \,  \Rgc(\Y/T_{\ell '},\Lambda) $}

\noindent where $T_{\ell'}$ is the $\ell'$-component of the abelian group $\T^{cF}$. The quotient variety $\Y/T_{\ell'}$ will be denoted by $\Y_\ell$. 

\sk

From now on, we shall work exclusively with $\Y_\ell$ instead of $\Y$ . This will not have any effect on the results we have deduced from $\mathrm{(W)}$ and $\mathrm{(S)}$ in \cite{Du3}. We start by proving an analog of \cite[Corollary 2.10]{Lu} for the variety $\Y_\ell$:

\begin{prop}\label{2prop2}Let $I$ be a maximal proper $\phi$-stable subset of $\Delta$. Then there exists an isomorphism of $\Lambda$-modules

\centers{$\Hc^{i}(U_I \backslash \Y_\ell, \Lambda) \, \simeq \, {}^*\mathrm{R}_{L_I}^G\big(\Hc^{i}(\Y_\ell, \Lambda)  \big) \, \simeq   \, \Hc^{i-1}(\X_I, \Lambda) \oplus \Hc^{i-2}(\X_I, \Lambda).$}

\end{prop}

\begin{proof}  Recall that the isomorphism $U_I \backslash \X \simeq \X_I \times \G_m$ constructed in \cite{Lu} is not $G$-equivariant. We can check that it is nevertheless $V_I$-equivariant, where $\V_I = \U\cap \L_I$ is a maximal unipotent subgroup of $\L_I$. We argue as in \cite{BR2}: the  isomorphism cannot be lifted up to $\Y_\ell$ but only to an abelian covering of this variety. Let $\Y^\circ$ be the preimage in $\Y$ of a connected component of $U \backslash \Y$. It is stable by the action of $U$ but not necessarily by the action of $\T^{cF}$ which permutes transitively the connected components of $U \backslash \Y$ (since $U \backslash \Y / \T^{cF} \simeq U \backslash \X$ is connected). If we denote by $H$ the stabilizer of the component $U\backslash \Y^\circ$ in $\T^{cF}$, we have the following $V_I \times (\T^{cF})^{\mathrm{op}}$-equivariant isomorphism
\begin{equation} \label{2eq1} U_I \backslash \Y^\circ \times_{H} \T^{cF} \, \simeq \, U_I \backslash \Y. \end{equation}
\noindent By Abhyankar's lemma, we can now lift the Galois covering  $U \backslash \Y^\circ \longrightarrow U \backslash \X \simeq (\G_m)^r$ up to a covering $\varpi : (\G_m)^r \longrightarrow (\G_m)^r$ with Galois group $\prod \mu_{m_i}$. Details of this construction are given in \cite{BR2} and \cite{Du1}. We define the variety $\widetilde \Y$ using  the following diagram in which all the squares are cartesian: 
\begin{equation}\label{2diag1} \hskip-4.5cm \begin{psmatrix}
[colsep=6mm]
[mnodesize=2cm]  & \widetilde{\Y}    & & &
 (\G_m)^r 
 &  \\[0pt]  [mnodesize=2cm]   & [mnodesize=2.4cm]  U_I \backslash \Y^\circ   & & &  U \backslash \Y^\circ  
 \\[0pt] [mnodesize=2cm]     & [mnodesize=2.4cm] U_I \backslash \X & & & U \backslash \X  & 
 [mnodesize=0pt] \hskip -3.8mm \simeq \ (\G_m)^r
\psset{arrows=->>,nodesep=3pt} 
\everypsbox{\scriptstyle} 
\ncline{1,5}{2,5}<{\varpi}>{/ N}  \ncarc[arcangle=25]{1,5}{3,6}>{/ \prod \mu_{m_i}}
\ncline{2,2}{2,5}  
 \ncline{3,2}{3,5}
 \ncline{2,2}{3,2}<{\pi^\circ}>{/H} \ncline{2,5}{3,5}<{\pi^\circ}>{/H}  
 \ncline{1,2}{2,2}>{/ N}
 \ncline{1,2}{1,5} 
\end{psmatrix}\end{equation}
\noindent Under the isomorphism  $U_I \backslash \X \simeq \X_I \times \G_m$ the map $U_I \backslash \X \longrightarrow (\G_m)^r$ decomposes into $\pi_I \times \mathrm{id}_{\G_m}$ where $\pi_I : \X_I \longrightarrow V_I \backslash \X_I \simeq (\G_m)^{r-1}$ is the quotient map associated with $\X_I$. In particular, we can form the following fiber product 
\begin{equation}\label{2diag2} \hskip-1.2cm  \begin{psmatrix}
[colsep=2.6cm]
 \widetilde{\Y}_I & (\G_m)^{r-1} \\ 
\X_I & (\G_m)^{r-1} 
\psset{arrows=->>,nodesep=3pt} 
\everypsbox{\scriptstyle} 
\ncline{1,1}{1,2} 
\ncline{1,1}{2,1}>{\big/ \mathop{\prod}\limits_{i\neq j} \mu_{m_i}} 
\ncline{1,2}{2,2}>{\big/ \mathop{\prod}\limits_{i\neq j}  \mu_{m_i}} 
\ncline{2,1}{2,2}^{\pi_I}
\end{psmatrix}\end{equation}
\noindent in order to decompose the variety $\widetilde \Y$ into a product $ \widetilde{\Y} \, \simeq \, \widetilde{\Y}_I \times \G_m$ which is compatible with the group decomposition $\prod \mu_{m_i} \, \simeq \, \big(\prod_{i \neq j} \mu_{m_i}\big) \times \mu_{m_j}$.   

\sk

Recall from \cite{BR2} that the cohomology of $\G_m$ endowed with the action of $\mu_m$ can be represented by a complex concentrated in degrees $1$ and $2$:

\centers{$0 \longrightarrow \Lambda \mu_m \mathop{\longrightarrow}\limits^{\zeta-1} \Lambda \mu_m \longrightarrow 0$} 

\noindent where $\zeta$ is any primitive $m$-th root of $1$. Such a complex will be denoted by $Z(\mu_m)[-1]$ so that the cohomology of $Z(\mu_m)$ vanishes outside the degrees $0$ and $1$. As in \cite[Section 3.3.2]{BR2}, we shall write $Z_{\T^{cF}}(\mu_m)$ for $Z(\mu_m) \otimes_{\Lambda \mu_m} \Lambda \T^{cF}$.

\sk

Using Proposition \ref{1prop1}, we can now express the previous constructions in cohomological terms. Equation \ref{2eq1} together with Diagram \ref{2diag1} leads to

\centers{$ \begin{array}{r@{\, \ \simeq \, \ }l} \Rgc(U_I \backslash \Y, \Lambda) & \Rgc(U_I \backslash \Y^\circ, \Lambda) \, {\ol}_{\Lambda H} \, \Lambda \T^{cF} \\[3pt] & \Rgc(\widetilde{\Y}, \Lambda) \, {\ol}_{\Lambda \prod   {\mu}_{m_i} } \, \Lambda \T^{cF}\end{array} $}

\noindent Using the isomorphism $\widetilde{\Y} \, \simeq \, \widetilde{\Y}_I \times \G_m$ and the notation that we have introduced, it can be written as 
\begin{equation} \begin{array}[c]{r@{\, \ \simeq \, \ }l}  \Rgc(U_I \backslash \Y, \Lambda) &\Big(\Rgc(\widetilde{\Y}_I, \Lambda) \, {\ol}_{\Lambda  \mathop{\prod}\limits_{i\neq j}  {\mu}_{m_i}} \, \Lambda \T^{cF} \Big) \, {\ol}_{\Lambda \T^{cF}} \, \Big( \Rgc(\G_m,\Lambda) \, {\ol}_{\Lambda  {\mu}_{m_j}} \, \Lambda \T^{cF} \Big) \\[4pt]
& \Big(\Rgc(\widetilde{\Y}_I, \Lambda) \, {\ol}_{\Lambda  \mathop{\prod}\limits_{i\neq j}  {\mu}_{m_i}} \, \Lambda \T^{cF} \Big) \, {\ol}_{\Lambda \T^{cF}} \, Z_{\T^{cF}} ({\mu}_{m_j})[-1] \end{array} \label{2eq2}\end{equation}
\noindent in the derived category $D^b(\mathrm{Mod}$-$\Lambda \T^{cF})$.

\sk

In general, the cohomology groups of the complex $Z_{\T^{cF}} ({\mu}_{m_j})$ are endowed with a non-trivial action of $\T^{cF}$ since the map ${\mu}_{m_j} \longrightarrow \T^{cF}$ is not surjective. One can actually compute explicitly the image of this map: by  \cite[Proposition 3.5]{BR2} it is precisely the group $N_c(Y_{c,c_I})$ (see \cite[Section 4.4.2]{BR1} for the definition). Therefore we obtain

\centers{$\mathrm{H}^0(Z_{\T^{cF}} ({\mu}_{m_j})) = \mathrm{H}^1(Z_{\T^{cF}} ({\mu}_{m_j})) \, \simeq \, \Lambda \big[\T^{cF} / N_c(Y_{c,c_I})\big].$}

\noindent By \cite[Proposition 4.4]{BR1}, the quotient $\T^{cF} / N_c(Y_{c,c_I})$ is isomorphic to $\T^{c_I F}/N_{c_I}(Y_{c,c_I})$ which has order prime to $\ell$. In particular, the image of the map ${\mu}_{m_j} \longrightarrow \T^{cF}$ contains the $\ell$-Sylow subgroup of $\T^{cF}$ and the map ${\mu}_{m_j} \longrightarrow \T^{cF}/T_{\ell'}$ is definitely onto. Since $Z_{\T^{cF}} ({\mu}_{m_j})$ fits into the following distinguished triangle in $D^b(\Lambda \T^{cF}$-$\mathrm{Mod})$ 

\centers{$ \Lambda\big[\T^{cF} / N_c(Y_{c,c_I}) \big] \longrightarrow  Z_{\T^{cF}} ({\mu}_{m_j}) \longrightarrow \Lambda\big[\T^{cF} / N_c(Y_{c,c_I})\big][-1] \rightsquigarrow $}

\noindent we deduce that the coinvariants under $T_\ell$ have a relatively simple shape
\begin{equation}\label{2eq3}  \Lambda \longrightarrow  Z_{\T^{cF}} ({\mu}_{m_j}) \, {\ol}_{\Lambda T_{\ell'}} \, \Lambda  \longrightarrow \Lambda[-1] \rightsquigarrow \end{equation}
\noindent Together with the expression of $\Rgc(\widetilde{\Y}_I, \Lambda) \, {\ol}_{\Lambda  \mathop{\prod}\limits_{i\neq j}  {\mu}_{m_i}} \, \Lambda \T^{cF}$ given in Formula \ref{2eq2}, this triangle yields

\centers{$ \Rgc(\widetilde{\Y}_I, \Lambda) \, {\ol}_{\Lambda  \mathop{\prod}\limits_{i\neq j}  {\mu}_{m_i}} \,\Lambda[-1] \longrightarrow \Rgc(U_I \backslash \Y_\ell,\Lambda) \longrightarrow\Rgc(\widetilde{\Y}_I, \Lambda) \, {\ol}_{\Lambda  \mathop{\prod}\limits_{i\neq j}  {\mu}_{m_i}} \, \Lambda[-2] \rightsquigarrow $}

\noindent and then simply

\centers{$ \Rgc(\X_I, \Lambda)[-1] \longrightarrow \Rgc(U_I \backslash \Y_\ell,\Lambda) \longrightarrow\Rgc(\X_I, \Lambda) [-2] \rightsquigarrow $}

\noindent by definition of $\widetilde \Y_I$.

\sk

We claim that the connecting maps $\Hc^{i-2}(\X_I,\Lambda) \longrightarrow \Hc^i(\X_I,\Lambda)$ coming from the long exact sequence associated with the previous triangle are actually zero. Since the triangle \ref{2eq3} splits over $K$, the scalar extension of these morphisms are zero. But by Proposition \ref{2prop1} the modules $\Hc^i(\X_I,\Lambda)$ are torsion-free, so that the connecting morphisms are indeed zero over $\Lambda$. Consequently, we obtain short exact sequences

\centers{$ 0 \longrightarrow \Hc^{i-1}(\X_I,\Lambda) \longrightarrow \Hc^{i}(\Y_\ell,\Lambda)^{U_I} \longrightarrow \Hc^{i-2}(\X_I,\Lambda) \longrightarrow 0 $}

\noindent which finishes the proof.\end{proof}

\begin{cor}\label{2cor2}The torsion part of any cohomology group $\Hc^i(\Y_\ell,\Lambda)$ is a cuspidal $\Lambda G$-module.
\end{cor}

As a consequence, Assumption $\mathrm{(S)}$ holds for $\Y_\ell$ if every cuspidal module in the block occurs in the cohomology group in middle degree only. By construction of $\Y_\ell$, it is actually sufficient to consider the cohomology of $\X$:

\begin{cor}\label{2cor3}Assume that one of the following holds:

\begin{enumerate} 

\item[$\mathrm{(1)}$] The cohomology of $\X$ over $\Lambda$ is torsion-free and Conjecture  \hyperref[2conj1]{\emph{(HLM)}} holds.

\item[$\mathrm{(2)}$] Cuspidal composition factors occur in $\Hc^i(\X,k)$ for $i=r$ only.

\end{enumerate}

\noindent Then the cohomology of $\Y_\ell$ over $\Lambda$ is torsion-free.
\end{cor}

\begin{proof} We have already mentioned at the end of Section \ref{2se2} that assertions $\mathrm{(1)}$ and $\mathrm{(2)}$ are equivalent. Denote by $\pi_\ell : \Y_\ell \longrightarrow \X$ the quotient map by the $\ell$-Sylow subgroup of $\T^{cF}$. The push-forward of the constant sheaf $\underline k_{\Y_\ell}$ on $\Y_\ell$ is obtained from  successive extensions of constant sheaves $\underline k_{\X}$:

\centers{$ (\pi_\ell)_* (\underline k_{\Y_\ell}) \, \simeq  \, 
\hskip-1.3mm \begin{array}{c} \underline  k_\X \\ \underline  k_\X \\ \vdots \\ \underline  k_\X \end{array} \hskip-1.3mm 
\ \ \cdot$}

\noindent In the derived category, the complex $\Rgc(\Y_\ell,k)  = \Rgc(\X,(\pi_\ell)_* (\underline k_{\Y_\ell}))$ can also be obtained from extensions of $\Rgc(\X,k)$'s. In other words, there exists a family of complexes $\Rgc(\Y_\ell,k) = C_0, C_1, \ldots, C_n = \Rgc(\X,k)$ that fit into the following distinguished triangles in $D^{b}(kG$-$\mathrm{Mod}$-$kT_\ell)$: 

\centers{$ C_{i+1} \longrightarrow C_{i} \longrightarrow \Rgc(\X,k) \rightsquigarrow $}

\noindent If we apply the functor $\mathrm{RHom}_{kG}^\bullet(P,-)$ to each of these triangles,  we can deduce that the complexes  $\mathrm{RHom}_{kG}^\bullet(P,C_i)$ are concentrated in degree $r$ whenever  $\mathrm{RHom}_{kG}^\bullet(P,C_n)$ is. Taking $P$ to be the projective cover of any cuspidal $kG$-module, we conclude that cuspidal composition factors of the cohomology of $\Y_\ell$ can only occur in middle degree. \end{proof}

\subsection{On the assumption \texorpdfstring{$\mathrm{(W)}$}{W} for \texorpdfstring{$\X$}{X}\label{2se4}}

We give here a proof of the analog of the assumption $\mathrm{(W)}$ for the Deligne-Lusztig variety  $\X$, using the smooth compactification $\overline \X$ constructed in \cite{DeLu}.

\begin{prop}\label{2prop3}If  $\lambda$ is an eigenvalue of $F^\delta$ on $\Hc^r(\X,K)$, then the $\Lambda$-modules $\Hc^i(\X,\Lambda)_{(\lambda)}$ are all torsion-free.
\end{prop}

\begin{proof} In  line with \cite{Lu}, we shall denote by $D=D_{r-1}(\Delta)$ the complement of $\X$ in $\overline{\X}$. Let $I$ be a $\phi$-stable subset of $\Delta$. By \cite[Section 7]{Lu}, the generalized $(\lambda)$-eigenspaces of $F^\delta$ on $\Hc^i(\X_I,K)$ are zero whenever $i \neq |I/\phi|$. By Proposition \ref{2prop1}, this remains true for cohomology groups with coefficients in $k$. From the long exact sequences
\begin{equation}\cdots \longrightarrow \displaystyle \hskip-2mm \bigoplus_{\begin{subarray}{c} I \ \phi\text{-stable} \\ |I/\phi| = a \end{subarray}} \hskip-2mm \mathrm{R}_{L_I}^G \big(\Hc^i(\X_I,k)\big) \longrightarrow \Hc^i(D_{a}(\Delta),k) \longrightarrow \Hc^i(D_{a-1}(\Delta),k) \longrightarrow \cdots \label{2eq4}
\end{equation}
\noindent we deduce that $\Hc^i(D,k)_{(\lambda)}$ is zero if $i \geq r$. Consequently, for all $i>r$ we have 

\centers{$ \Hc^i(\X,k)_{(\lambda)} \, \simeq \, \mathrm{H}^i(\overline{\X},k)_{(\lambda)}.$}

The eigenvalue $\mu = q^{2r} \lambda^{-1}$  is a "maximal" eigenvalue of $F^\delta$ and as such does not occur in the cohomology of the varieties $\X_I$ for proper subsets $I$. From the previous long exact sequences we deduce that the generalized $\bar \mu$-eigenspace of $F^\delta$ on $\Hc^i(D,k)$ is always zero. Since $\overline \X$ is a smooth variety, we can apply Poincar\'e duality (see Thereom \ref{1thm2}) in order to obtain the following isomorphisms: 

\centers{$ \Hc^i(\X,k)_{(\lambda)} \, \simeq \, 
 \mathrm{H}^i(\overline{\X},k)_{(\lambda)}  
\, \simeq \, \big(\mathrm{H}^{2r-i}(\overline{\X},k)_{(\mu)}\big)^* \, \simeq \, \big(\Hc^{2r-i}(\X,k)_{(\mu)}\big)^*.$}

\noindent Since $\X$ is an irreducible affine variety, the cohomology of $\X$ vanishes outside the degrees $r,\ldots,2r$. This proves that $\Hc^{i}(\X,k)_{(\lambda)}=0$ whenever $i > r$, and the result follows from the universal coefficient theorem.
\end{proof}

\begin{rmk} One cannot deduce that the assumption $\mathrm{(W)}$ holds for $\Y_\ell$ using this result. The method that we used in the proof of Corollary \ref{2cor3} does not preserve the generalized eigenspaces of $F^\delta$. 

\sk

Note however that Bonnaf\'e and Rouquier have constructed in \cite{BR4} a compactification $\overline \Y_\ell$ of $\Y_\ell$ such that $\overline \Y_\ell \smallsetminus \Y_\ell \simeq \overline{\X} \smallsetminus \X$. This compactification is only rationally smooth in general, but if it happened to  be $k$-smooth, the previous method would extend to $\Y_\ell$ and finish the proof of the conjecture of Hi\ss-L\"ubeck-Malle.
\end{rmk}

\section{Cuspidal composition factors in \texorpdfstring{$\Hc^\bullet(\X,k)$}{Hc(X)}\label{3se}}

We have reduced the proof of the assumption $\mathrm{(S)}$ to showing that the $kG$-modules $\Hc^i(\X,k)$ have no cuspidal composition factors unless $i=r$. In other words, the cohomology of $\mathrm{RHom}_{k G}^\bullet\big(\overline{P}_L,\Rgc(\X,k)\big)$ should vanish outside the degree $r$ whenever $L$ is cuspidal.

\sk

Throughout this section and unless otherwise specified, we shall alway assume that Conjecture  \hyperref[2conj1]{(HLM)} holds for $(\G,F)$ (which is known to be true except for the groups of type E$_7$ or E$_8$). In that case, the above result can be deduced from the characteristic zero case if we can show that the cohomology of $\mathrm{RHom}_{\Lambda G}^\bullet\big(P_L,\Rgc(\X,\Lambda)\big)$ is torsion-free. Indeed, the irreducible unipotent components of $[P_L]$ are the characters $\chi_{m_\zeta}$'s, and they occur in the cohomology of $\X$ in middle degree only. We shall divide the proof according to the depth of $\chi_{m_\zeta}$: if $\chi_{m_\zeta}$ is cuspidal, we prove that the contribution of $P_L$ to the cohomology of $\X$ and its compactification $\overline{\X}$ are the same. In the case where $\chi_{m_\zeta}$ is not cuspidal, it is induced from the cohomology of a smaller Coxeter variety and we show that the contribution of $P_L$ can be computed in terms of this cohomology, which we know to be torsion-free.

\subsection{Harish-Chandra series of length 1\label{3se1}}

We start by dealing with the case of cuspidal $kG$-modules that occur as $\ell$-reductions of cuspidal unipotent characters. By Conjecture  \hyperref[2conj1]{(HLM)}, these correspond to subtrees of the Brauer tree of the form

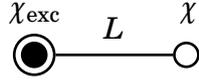
\begin{figure}[h] 
\centers[3]{\begin{pspicture}(2,0.3)
  \psset{linewidth=1pt}

  \cnode[fillstyle=solid,fillcolor=black](0,0){5pt}{A2}
  \cnode(0,0){8pt}{A}
  \cnode(2,0){5pt}{B}
  
  \ncline[nodesep=0pt]{A}{B}\naput[npos=-0.23]{$\phantom{\bigg(}\chi_{\mathrm{exc}}$}\naput[npos=1.08]{$\phantom{\bigg(}\chi$}\naput{$L$}

\end{pspicture}}
\end{figure}

\noindent where $\chi$ is a cuspidal unipotent character and $L$ is the unique composition factor of the $\ell$-reduction of $\chi$. By \cite[Proposition 4.3]{Lu}, the character $\chi$ occurs only in the cohomology of $\X$ in middle degree. The following is a modular analog of this result.

\begin{prop}\label{3prop1}Assume that Conjecture  \hyperref[2conj1]{\emph{(HLM)}} holds for $(\G,F)$. Let $\chi$ be a cuspidal character of $G$ occurring in $\Hc^r(\X,K)$. The $\ell$-reduction of $\chi$ gives a unique simple cuspidal $kG$-module and it occurs as a composition factor of $\Hc^i(\X,k)$ for $i=r$ only.
\end{prop}

\begin{proof} We keep the notation used in the course of the proof of Proposition \ref{2prop3}: $D$ is the complement of $\X$ in the compactification $\overline \X$. Recall that the cohomology of $D$ can be computed in terms of induced characters afforded by the cohomology of smaller Coxeter varieties. Since $\chi$ is cuspidal, we can therefore use the long exact sequences \ref{2eq4} to show that $\chi$ does not occur in $\Hc^\bullet(D,K)$. 

\sk

Denote by $L$  the unique composition factor of the $\ell$-reduction of $\chi$. By Conjecture  \hyperref[2conj1]{(HLM)}, $L$ does not appear in the $\ell$-reduction of any other unipotent character of the block. Therefore, using Proposition \ref{2prop1} and the long exact sequences \ref{2eq4} again, we deduce that $L$ cannot occur as a composition factor of $\Hc^\bullet(D,k)$. In particular, if we denote by $\overline{P}_L \in kG$-$\mathrm{mod}$ the projective cover of $L$ then 

\centers{$ \mathrm{Hom}_{kG}\big(\overline{P}_L, \Hc^i(\X,k)\big) \, \simeq \, \mathrm{Hom}_{kG}\big(\overline{P}_L, \Hc^i(\overline \X,k)\big).$}

\noindent Moreover, since duality preserve cuspidality, the same argument applies for the dual of the cohomology of $D$. Subsequently, using Poincaré duality, we obtain

\centers{$  \mathrm{Hom}_{kG}\big(\overline{P}_L, \Hc^i(\overline \X,k)\big) \, \simeq \,  \mathrm{Hom}_{kG}\big(\overline{P}_L, \Hc^{2r-i}(\overline \X,k)^*\big) \, \simeq \,  \mathrm{Hom}_{kG}\big(\overline{P}_L, \Hc^{2r-i}(\X,k)^*\big).$}

\noindent Now the cohomology of the irreducible affine variety $\X$  vanishes outside the degrees $r,\ldots,2r$. It follows that $ \mathrm{Hom}_{kG}\big(\overline{P}_L, \Hc^i(\X,k)\big) = 0$ whenever $i>r$.
\end{proof}

\subsection{Induced characters\label{3se2}}

We now turn to the case of cuspidal $kG$-modules occurring as composition factors of the $\ell$-reduction of a non-cuspidal unipotent character. If we assume that Conjecture  \hyperref[2conj1]{(HLM)} holds, then by the results of \cite{Du3} the situation corresponds to branches of the following form:

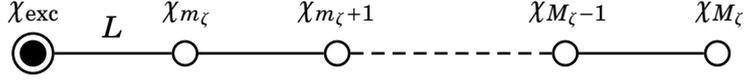
\centers[3]{\begin{pspicture}(9,0.8)
  \psset{linewidth=1pt}

  \cnode[fillstyle=solid,fillcolor=black](0,0){5pt}{A2}
  \cnode(0,0){8pt}{A}
  \cnode(2,0){5pt}{B}
  \cnode(4,0){5pt}{C}
    \cnode(7,0){5pt}{D}
  \cnode(9,0){5pt}{E}

  \ncline[nodesep=0pt]{A}{B}\naput[npos=-0.23]{$\phantom{\bigg(}\chi_{\mathrm{exc}}$}\naput[npos=1.08]{$\phantom{\bigg(}\chi_{m_\zeta}$}\naput{$L$}
  \ncline[nodesep=0pt]{B}{C}
  \ncline[nodesep=0pt,linestyle=dashed]{C}{D}\naput[npos=-0.10]{$\phantom{\bigg(}\chi_{m_\zeta+1}$}\naput[npos=1.05]{$\phantom{\bigg(}\chi_{M_\zeta-1}$}
    \ncline[nodesep=0pt]{D}{E}\naput[npos=1.08]{$\phantom{\bigg(}\chi_{M_\zeta}$}
\end{pspicture}}

\noindent where $m_\zeta < M_\zeta$. 

\sk

At the level of characters, $\chi_{m_\zeta}$ is obtain by inducing a cuspidal unipotent character occurring in the cohomology of a smaller Coxeter variety $\X_I$. The projective cover of $L$ can in turn be obtained by a suitable induction using the following result. 

\begin{lem} Let $I$ be an $F$-stable subset of $\Delta$, and $J$ be its complement. Let $\V_I = \L_I \cap \U$ (resp. $\V_J = \L_J \cap \U$) be the maximal unipotent subgroup of $\L_I$ (resp. $\L_J$) contained in $\U$. Then there is an isomorphism of varieties, compatible with the actions of $F^\delta$ and $V_I \times V_J$

\centers{$ (U_I \cap U_J) \backslash \X \, \simeq \, \X_I \times \X_J$.}

\end{lem}

\begin{proof} For $\alpha \in \Phi^+$ we denote by $\U_\alpha$ the corresponding one-parameter subgroup of $\U$ and we put $\U_\alpha^\sharp = \U_\alpha \smallsetminus \{1\}$. By \cite[Theorem 2.6]{Lu}, there is a $U$-equivariant isomorphism of varieties

\centers{$ \X \, \simeq \, \Big\{ u \in \U \ \big| \ u^{-1}F(u) \in \displaystyle \prod_{\alpha \in I} \U_{\alpha}^\sharp \prod_{\beta \in J} \U_{\beta}^\sharp \Big\}.$}

\noindent The method used in \cite[Proposition 1.2]{Du1} to describe the quotient of $\B$ by the group $D(\U)^F$ extends to any unipotent normal subgroup of $\B$ normalized by $\T$  instead of $D(\U)$ (see \cite[Section 2.3.2]{Du2} for more details). Using the isomorphism  $(\U_I \cap \U_J) \backslash \U \simeq \V_I \times \V_J$, we can therefore realize the quotient of the Deligne-Lusztig variety by $U_I \cap U_J$ as: 

\centers{$ (U_J \cap U_J) \backslash \X \, \simeq \, \left\{ (\bar u, v_1,v_2) \in \U \times \V_I \times \V_J \ \left| \begin{array}{l}  \pi_{\U_I\cap \U_J}(\bar u) = v_1^{-1} F(v_1) v_2^{-1} F(v_2) \\[5pt] 
\bar u \in  \displaystyle \prod_{\alpha \in I} \U_\alpha^\sharp \prod_{\beta \in J} \U_\beta^\sharp 
\end{array} \right.\right\}. $}

\noindent This can be rephrased in the  $V_I\times V_J$-equivariant isomorphism
$V \backslash \X \, \simeq \, \X_{\L_I}(c_I) \times \X_{\L_J} (c_J) \, \simeq \, \X_I \times \X_J.$
\end{proof}

Recall that a \emph{regular} character $V_J$ is any linear character of $V_J$, trivial on $D(\V_J)^F$ such that the induced character on the abelian group $V_J/D(\V_J)^F$ is a product of non-trivial characters (see \cite[Section 2]{BR2} for more details). If $\psi$ is such a character, and $\Lambda_\psi$ is the corresponding one-dimensional $\Lambda V_J$-module we define

\centers{$ \Gamma_I \, = \, \mathrm{Ind}_{U}^G \, \mathrm{Res}_U^{V_I \times V_J} \, (\Lambda V_I \otimes_\Lambda \Lambda_\psi)$}

\noindent where  the restriction is taken through  the map $U \longrightarrow U / U_I\cap U_J \simeq V_I \times V_J$.



\begin{lem} Let $I$ to be the minimal $F$-stable subset of $\Delta$ such that ${}^* \mathrm{R}_{L_I}^G (\chi_{m_\zeta}) \neq 0$. Then $P_L$ is a direct summand of  $\Gamma_I$.
\end{lem}

\begin{proof} Since $\Gamma_I$ is projective, it is enough to show that $\chi_{m_\zeta}$ is the only constituent of $[\Gamma_I]$ in the Harish-Chandra series corresponding to $\zeta$ (cut by the principal block). To this end, we can compute the contribution of $\Gamma_I$ to the cohomology of the Coxeter variety using the previous lemma.

\centers{$ \begin{array}{r@{\ \, \simeq \ \,}l} \mathrm{RHom}_{\Lambda G}^\bullet(\Gamma_I, \Rgc(\X,\Lambda)) & \mathrm{RHom}_{\Lambda U}^\bullet\big( \mathrm{Res}_U^{V_I \times V_J} \, (\Lambda V_I \otimes_\Lambda \Lambda_\psi), \Rgc(\X,\Lambda)\big) \\[6pt]
 & \mathrm{RHom}_{\Lambda (V_I \times V_J) }^\bullet\big( \Lambda V_I \otimes_\Lambda \Lambda_\psi, \Rgc((U_I \cap U_J) \backslash \X,\Lambda)\big) \\[6pt]
 & \Rgc(\X_I,\Lambda) \otimes_\Lambda \mathrm{RHom}_{\Lambda V_J}^\bullet(\Lambda_\psi,\Rgc(\X_J,\Lambda)). \end{array}$}

\noindent From  \cite[Theorem 3.10]{BR2} we deduce that the complex  $\mathrm{RHom}_{\Lambda V_J}^\bullet(\Lambda_\psi,\Rgc(\X_J,\Lambda))$ is quasi-isomorphic to the trivial module shifted in degree $r_J = \dim \X_J$ and we obtain finally
\begin{equation}
\mathrm{RHom}_{\Lambda G}^\bullet(\Gamma_I, \Rgc(\X,\Lambda)) \, \simeq \, \Rgc(\X_I,\Lambda)[-r_J].
\label{contributioneq}
\end{equation}
\indent Tensoring this quasi-isomorphism with $K$ gives the contribution of the character of $\Gamma_I$ to the graded character afforded by the cohomology of $\X$. In particular, a character $\chi$ in the series associated to $\zeta$ and in the principal $\ell$-block occurs in $\Gamma_I$ if and only if the corresponding eigenvalue of $F^\delta$ occurs in $\Hc^\bullet(\X_I,K)$. Since $I$ is chosen to be minimal for $\zeta$, then by  \cite[Theorem 7.1]{Lu} this eigenvalue is necessarily $\zeta q^{\delta \dim \X_I / 2}$ and it occurs in degree $r= \dim \X = \dim \X_I + \dim \X_J$ only. This forces $\chi$ to be the unique character of the series occuring in $\Hc^r(\X,K)$, that is $\chi_{m_\zeta}$.

\sk

Finally, using the shape of the Brauer tree, any projective indecomposable module that has $\chi_{m_\zeta}$ as its only constituent in the Harish-Chandra series associated to $\zeta$ must be the projective cover of $L$. 
\end{proof}

\begin{prop}\label{3cor1}Assume that Conjecture  \hyperref[2conj1]{\emph{(HLM)}} holds. Let $L$ be the unique simple cuspidal $kG$-module that occur in any $\ell$-reduction of $\chi_{m_\zeta}$. Then $L$ occurs as a composition factor of $\Hc^i(\X,k)$ for $i=r$ only.
\end{prop}

\begin{proof} Let $I$ be minimal such that ${}^* \mathrm{R}_{L_I}^G (\chi_{m_\zeta}) \neq 0$.  By  \ref{contributioneq}  and Proposition \ref{2prop1} the cohomology of $\mathrm{RHom}_{\Lambda G}^\bullet\big(\Gamma_I,\Rgc(\X,\Lambda)\big)$ is torsion-free. Since $P_L$ is a direct summand of $\Gamma_I$ the same holds for $P_L$. Consequently, the cohomology of $\mathrm{RHom}_{\Lambda G}^\bullet\big(P_L,\Rgc(\X,\Lambda)\big)$ vanishes outside the degree $r$ since it already does over $K$ by \cite[Proposition 4.3]{Lu}.
\end{proof}

\subsection{Main results\label{3se3}}

We now have all the ingredients for proving that there is no torsion in the cohomology of $\Y_\ell$. Recall that we have shown that the torsion part of the cohomology is necessarily a cuspidal $\Lambda G$-module (see Corollary \ref{2cor2}). By the universal coefficient theorem, it is therefore sufficient to prove that the module $\Hc^i(\Y_\ell,k)$ has no cuspidal composition factors if $i > r$. By Corollary \ref{2cor3}, this property holds whenever it holds for the Deligne-Lusztig variety $\X$. In the framework of derived categories, we are then reduced to show that for any cuspidal $kG$-module $L$ lying in the block, the complex

\centers{$ \mathrm{RHom}_{\Lambda G}^\bullet \big(\overline{P}_L, \Rgc(\X,k)\big) $}

\noindent is quasi-isomorphic to a complex concentrated in degree $r$. If $L$ happens to be a composition factor of the $\ell$-reduction of a cuspidal unipotent character, then by Proposition \ref{3prop1} it cannot occur outside the cohomology in middle degree. Corollary \ref{3cor1} deals with the case where $L$ is a composition factor of the $\ell$-reduction of an induced character, but we have to assume that Conjecture  \hyperref[2conj1]{(HLM)} holds.

\begin{thm}\label{3thm1}Let $\G$ be a quasi-simple group. Assume that Conjecture  \hyperref[2conj1]{(HLM)} holds. Then in the set-up of Section \ref{2se1}, the $\Lambda$-modules $b\Hc^i(\Y_\ell,\Lambda)$ are torsion-free.
\end{thm}

We deduce from \cite[Theorem 4.12]{Du3} the geometric version of Brou\'e's conjecture holds for any quasi-simple group except $E_7$ and $E_8$. This extends significantly the previous results of Puig \cite{Pu2} (for $\ell \, | \, q-1$), Rouquier \cite{Rou} (for $\ell \, | \, \phi_h(q)$ and $r=1$) and Bonnaf\'e-Rouquier \cite{BR2} (for $\ell \, |\, \phi_h(q)$ and $(\G,F)$ of type A$_n$).

\begin{thm}\label{3thm2}Let $(\G,F)$ be a quasi-simple group different from $E_7$ and $E_8$ and $c$ be a Coxeter element of $(W,F)$. Let  $\ell$ be a prime number not dividing the order of $W^F$ and satisfying one of the two following assumptions, depending on the type of $(\G,F)$: 

\begin{itemize}
 
 \item[$\bullet$] "non-twisted" cases: $\ell$ divides $\Phi_h(q)$; 

 \item[$\bullet$] "twisted" cases: $\ell$ divides the order of $T_c$.
 
\end{itemize}

\noindent Then the complex $b\Rgc(\Y(\dot c),\Lambda)b'$ induces a splendid and perverse equivalence between the principal $\ell$-blocks $b\Lambda G$ and $b' \Lambda N_G(T_c)$
\end{thm}

Using \cite[Theorem 4.14]{Du3} we also deduce the planar embedding of the Brauer tree of the principal $\ell$-block for groups of type ${}^2$G$_2$, ${}^2$F$_4$ and F$_4$ (compare with \cite{Hiss} and \cite{HL}).

\begin{thm}\label{3thm3}$\mathrm{(i)}$ Assume that $\ell$ is odd and divides $q^2-q\sqrt3 +1$. Then the planar embedded Brauer tree of the principal $\ell$-block of the Ree group ${}^2$G$_2(q^2)$ is

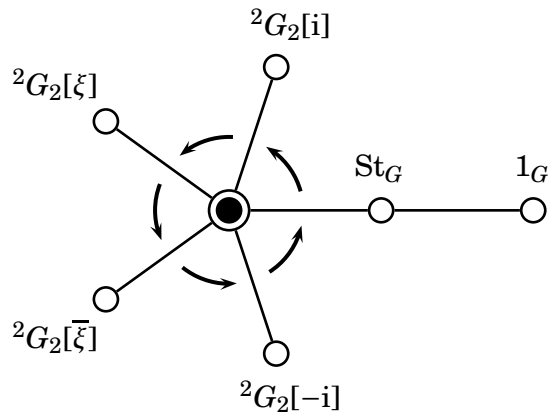
\centers[3]{\begin{pspicture}(6,4.5)

  \psset{linewidth=1pt}

  \cnode[fillstyle=solid,fillcolor=black](2,2){5pt}{A2}
  \cnode(2,2){8pt}{A}
  \cnode(4,2){5pt}{B}
  \cnode(6,2){5pt}{C}
  \cnode(2.62,3.9){5pt}{D}
  \cnode(2.62,0.1){5pt}{E}
  \cnode(0.38,3.18){5pt}{F}
  \cnode(0.38,0.82){5pt}{G}

  \ncline[nodesep=0pt]{A}{B}\naput[npos=1.1]{$\vphantom{\Big(} \mathrm{St}_G$}
  \ncline[nodesep=0pt]{B}{C}\naput[npos=1.1]{$\vphantom{\Big(} \mathrm{1}_G$}
  \ncline[nodesep=0pt]{A}{D}\ncput[npos=1.5]{${}^2\text{G}_2[\mathrm{i}]$}
  \ncline[nodesep=0pt]{A}{E}\ncput[npos=1.5]{${}^2\text{G}_2[-\mathrm{i}]$}
  \ncline[nodesep=0pt]{A}{F}\ncput[npos=1.65]{${}^2\text{G}_2[\xi]$}
  \ncline[nodesep=0pt]{A}{G}\ncput[npos=1.65]{${}^2\text{G}_2[\overline{\xi}]$}
  \psellipticarc[linewidth=1.5pt]{->}(2,2)(1,1){15}{60}
  \psellipticarc[linewidth=1.5pt]{->}(2,2)(1,1){87}{132}
  \psellipticarc[linewidth=1.5pt]{->}(2,2)(1,1){159}{204}
  \psellipticarc[linewidth=1.5pt]{->}(2,2)(1,1){231}{276}
  \psellipticarc[linewidth=1.5pt]{->}(2,2)(1,1){303}{348}

\end{pspicture}}

\bk

\noindent where $\mathrm{i} = \xi^3$ and $\xi$ is the unique $12$-th root of unity in $\Lambda^\times$ congruent to $q^5$ modulo $\ell$.

\mk

\noindent $\mathrm{(ii)}$ Assume that $\ell$ divides $q^4-\sqrt 2 q^3+q^2-\sqrt 2q+1$. Then the planar embedded Brauer tree of the principal $\ell$-block of the simple group of type ${}^2$F$_4(q)$ is

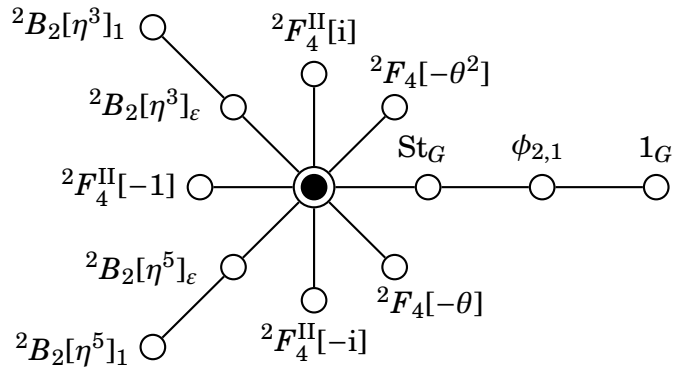
\centers[4]{\begin{pspicture}(6,4.5)

  \cnode[fillstyle=solid,fillcolor=black](1.5,2){5pt}{A2}
    \cnode(1.5,2){8pt}{A}
  \cnode(3,2){5pt}{B}
  \cnode(4.5,2){5pt}{C}
  \cnode(6,2){5pt}{D}
  \cnode(0,2){5pt}{O}
  \cnode(2.56,3.06){5pt}{P}
  \cnode(1.5,3.5){5pt}{Q}
  \cnode(0.44,3.06){5pt}{R}  \cnode(-.62,4.12){5pt}{R2}
  \cnode(2.56,0.94){5pt}{S}
  \cnode(1.5,0.5){5pt}{T}
  \cnode(0.44,0.94){5pt}{U}\cnode(-.62,-0.12){5pt}{U2}

  \ncline[nodesep=0pt]{A}{B}\naput[npos=1.1]{$\vphantom{\big(_p} \mathrm{St}_G$}
  \ncline[nodesep=0pt]{B}{C}\naput[npos=1.1]{$\vphantom{\Big(}\phi_{2,1}$}
  \ncline[nodesep=0pt]{C}{D}\naput[npos=1.15]{$\vphantom{\big(_p}  1_G$}
  \ncline[nodesep=0pt]{O}{A}\ncput[npos=-1.2]{$\vphantom{\Big(} {}^2F_{4}^{\mathrm{II}}[-1] $}
 \ncline[nodesep=0pt]{A}{P} \ncput[npos=1.8]{${}^2 F_4[-\theta^2]$}
 \ncline[nodesep=0pt]{A}{Q}\ncput[npos=1.7]{${}^2 F_4^{\mathrm{II}}[\mathrm{i}]$}
 \ncline[nodesep=0pt]{A}{R}\ncput[npos=1.2]{${}^2 B_2[\eta^3]_\varepsilon$ \hphantom{aaaiaaaaaa}} 
  \ncline[nodesep=0pt]{R}{R2}\ncput[npos=1.2]{${}^2 B_2[\eta^3]_1$ \hphantom{aaaaaaaaa}} 

 \ncline[nodesep=0pt]{A}{S}\ncput[npos=1.8]{${}^2 F_4[-\theta]$}
 \ncline[nodesep=0pt]{A}{T}\ncput[npos=1.7]{${}^2 F_4^{\mathrm{II}}[-\mathrm{i}]$}
 \ncline[nodesep=0pt]{A}{U}\ncput[npos=1.2]{${}^2 B_2[\eta^5]_\varepsilon$  \hphantom{aaaaaaaaaa}} 
  \ncline[nodesep=0pt]{U}{U2}\ncput[npos=1.2]{${}^2 B_2[\eta^5]_1$  \hphantom{aaaaiaaaa}}

\end{pspicture}}

\noindent where $\theta$ (resp. $\mathrm{i}$, resp. $\eta$) the unique primitive $3$-rd  (resp. $4$-th, resp. $8$-th) root of unity in $\Lambda^\times$  congruent to $q^8$ (resp. $q^6$, resp. $q^3$) modulo $\ell$.

\mk

\noindent $\mathrm{(iii)}$ Assume that $\ell \neq 2,3$ and divides $q^4-q^2 +1$. Then the planar embedded Brauer tree of the principal $\ell$-block of the simple group of type F$_4(q)$ is

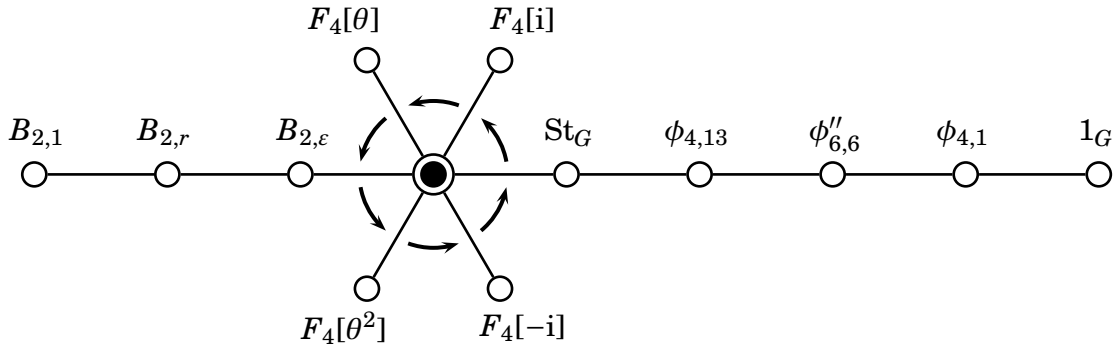
\centers[3]{\begin{pspicture}(14.1,4)

  \psset{linewidth=1pt}

  \cnode[fillstyle=solid,fillcolor=black](5.25,2){5pt}{A2}
  \cnode(5.25,2){8pt}{A}
  \cnode(7,2){5pt}{B}
  \cnode(8.75,2){5pt}{C}
  \cnode(10.5,2){5pt}{D}
  \cnode(12.25,2){5pt}{E}
  \cnode(14,2){5pt}{F}
  \cnode(3.5,2){5pt}{G}
  \cnode(1.75,2){5pt}{H}
  \cnode(0,2){5pt}{I}
  \cnode(6.125,3.515){5pt}{M}
  \cnode(6.125,0.485){5pt}{N}
  \cnode(4.375,3.515){5pt}{O}
  \cnode(4.375,0.485){5pt}{P}

  \ncline[nodesep=0pt]{A}{B}
  \ncline[nodesep=0pt]{B}{C}\naput[npos=-0.1]{$\vphantom{\Big(} \mathrm{St}_G$}\naput[npos=1.1]{$\vphantom{\Big(} \phi_{4,13}$}
  \ncline[nodesep=0pt]{C}{D}
  \ncline[nodesep=0pt]{D}{E}\naput[npos=-0.1]{$\vphantom{\Big(} \phi_{6,6}''$}\naput[npos=1.1]{$\vphantom{\Big(} \phi_{4,1}$}
  \ncline[nodesep=0pt]{E}{F}\naput[npos=1.1]{$\vphantom{\Big(} 1_G$}  
  \ncline[nodesep=0pt]{I}{H}\naput[npos=-0.1]{$\vphantom{\Big(} B_{2,1}$}\naput[npos=1.1]{$\vphantom{\Big(} B_{2,r}$}
  \ncline[nodesep=0pt]{H}{G}
  \ncline[nodesep=0pt]{G}{A}\naput[npos=-0.1]{$\vphantom{\Big(} B_{2,\varepsilon}$}
  \ncline[nodesep=0pt]{M}{A}\ncput[npos=-0.6]{$F_4[\mathrm{i}]$}
  \ncline[nodesep=0pt]{N}{A}\ncput[npos=-0.6]{$F_4[-\mathrm{i}]$}
  \ncline[nodesep=0pt]{O}{A}\ncput[npos=-0.6]{$F_4[\theta]$}
  \ncline[nodesep=0pt]{P}{A}\ncput[npos=-0.6]{$F_4[\theta^2]$}

  \psellipticarc[linewidth=1.5pt]{->}(5.25,2)(1,1){10}{50}
  \psellipticarc[linewidth=1.5pt]{->}(5.25,2)(1,1){70}{110}
  \psellipticarc[linewidth=1.5pt]{->}(5.25,2)(1,1){130}{170}
  \psellipticarc[linewidth=1.5pt]{->}(5.25,2)(1,1){190}{230}
  \psellipticarc[linewidth=1.5pt]{->}(5.25,2)(1,1){250}{290}
  \psellipticarc[linewidth=1.5pt]{->}(5.25,2)(1,1){310}{350}

\end{pspicture}}

\noindent where $\theta$ (resp. $\mathrm{i}$) is the unique third (resp. fourth) root of unity in $\Lambda^\times$ congruent to $q^{4}$ (resp. $q^3$) modulo $\ell$.

\end{thm}

\begin{center} {\subsection*{Acknowledgements}} \end{center}

The author wishes to thank Meinolf Geck and George Lusztig for useful correspondence and C\'edric Bonnaf\'e, Jean-Fran\c{c}ois Dat and Rapha\"el Rouquier for many valuable comments and suggestions.

\bk

\bibliographystyle{abbrv}
\bibliography{coxorbits2}

\end{document}